\documentclass[11pt]{article}
\usepackage[a4paper,margin=1in]{geometry}
\usepackage{amsmath,amssymb,amsthm,mathtools}
\usepackage{enumitem}
\usepackage{microtype}
\usepackage{indentfirst}
\usepackage[numbers,sort&compress]{natbib}
\usepackage[hidelinks]{hyperref}
\hypersetup{
  pdftitle={Spectral Erdos-Gallai Theorems for the A-alpha-Tensor of the s-Clique Hypergraph},
  pdfauthor={Xiaoqi Liu and Haiying Shan},
  pdfsubject={Extremal alpha-s-clique spectral radius under a matching constraint},
  pdfkeywords={clique tensor, A-alpha tensor, alpha-s-clique spectral radius, matching, generalized Turan number}
}
\newcommand{\doi}[1]{\url{https://doi.org/#1}}
\newtheorem{theorem}{Theorem}[section]
\newtheorem{lemma}[theorem]{Lemma}
\newtheorem{proposition}[theorem]{Proposition}
\theoremstyle{definition}
\newtheorem{definition}[theorem]{Definition}
\newtheorem{problem}[theorem]{Open Problem}
\theoremstyle{remark}
\newtheorem{remark}[theorem]{Remark}

\newcommand{\Cs}{\mathcal C_s}          
\newcommand{\As}{\mathcal A_s}          
\newcommand{\Ds}{\mathcal D_s}          
\newcommand{\Aas}{\mathcal A_{\alpha,s}}
\newcommand{\rhoas}{\rho_{\alpha,s}}    
\newcommand{\Fstar}{F^{t,t-1}(n)}       
\newcommand{\Kstar}{K_{2t-1}\cup\overline K_{n-2t+1}} 

\title{Spectral Erd\H{o}s--Gallai Theorems
for the \(\mathcal A_\alpha\)-Tensor of the \(s\)-Clique Hypergraph}
\author{Xiaoqi Liu, \, Haiying Shan\thanks{Corresponding author. E-mail:
\texttt{2111152@tongji.edu.cn} (X.~Liu),
\texttt{shan\_haiying@tongji.edu.cn} (H.~Shan). This work is partially
supported by the National Natural Science Foundation of China
(No.~12271182).}
\\ \small School of Mathematical Sciences, Key Laboratory of Intelligent
Computing and Applications\\ \small (Ministry of Education), Tongji
University, Shanghai, China}
\date{}

\begin{document}
\markboth{Erd\H{o}s--Gallai for the $\mathcal A_\alpha$-tensor of the $s$-clique hypergraph}
         {Erd\H{o}s--Gallai for the $\mathcal A_\alpha$-tensor of the $s$-clique hypergraph}
\maketitle

\begin{abstract}
The Erd\H{o}s--Gallai theorem determines the maximum number of edges in a graph with bounded matching number; its clique-counting extension replaces edges by $s$-cliques, and a spectral analogue in terms of the $s$-clique tensor has recently been established. We study the corresponding $\mathcal A_\alpha$-tensor of the $s$-uniform clique hypergraph. For $0\leq\alpha\leq1$ and $3\le s\le2t-1$, we determine the maximum $\alpha$-$s$-clique spectral radius among $n$-vertex graphs containing no matching of $t$ edges: when $3\le s\le t$ and $n$ is sufficiently large, the maximum is attained by the join of a clique of order $t-1$ and an independent set, and when $t<s\le 2t-1$ and $n\ge2t-1$, it is attained by a clique of order $2t-1$ together with isolated vertices. At $\alpha=0$, these statements recover the known result for the $s$-clique spectral radius; at $\alpha=1$, they yield the corresponding statement for the maximum $s$-clique degree. For $s=t\ge3$, we obtain the maximum for every $n\ge2t-1$; at $\alpha=0$, this removes the requirement that $n$ be sufficiently large from the known result.
\end{abstract}


\section{Introduction}
For a graph $F$, let $\mathrm{ex}(n,F)$ be the maximum number of edges in an
$n$-vertex graph containing no subgraph isomorphic to $F$; such a graph is
called $F$-free. Tur\'an~\cite{Turan} determined $\mathrm{ex}(n,K_{r+1})$,
the Tur\'an number, whose extremal graphs are the Tur\'an graphs. The
generalized Tur\'an number $\mathrm{ex}(n,H,F)$ counts copies of a fixed graph
$H$ instead of edges.

Zykov~\cite{Zykov} and Erd\H{o}s~\cite{ErdosCliques} determined
$\mathrm{ex}(n,K_s,K_{r+1})$ for $2\leq s\leq r$; Alon and
Shikhelman~\cite{AlonShikhelman} initiated the systematic study of
$\mathrm{ex}(n,H,F)$ for general pairs $H,F$. Let $M_t$ denote a matching of
$t$ pairwise disjoint edges. Erd\H{o}s and Gallai~\cite{ErdosGallai}
determined $\mathrm{ex}(n,M_t)$, and Wang~\cite{Wang} determined the
corresponding generalized Tur\'an number $\mathrm{ex}(n,K_s,M_t)$; we refer
to~\cite{Wang} also for the connection with the Erd\H{o}s matching conjecture.

Spectral extremal graph theory replaces the number of edges by the adjacency
spectral radius $\rho(G)$. Nosal~\cite{Nosal} proved the spectral Mantel
theorem, and Nikiforov~\cite{Nikiforov2007} extended it to
$K_{r+1}$-free graphs. For matchings, Feng, Yu, and
Zhang~\cite{FengYuZhang} determined the maximum adjacency spectral radius of
graphs with a given matching number. More recently, Yan, Yang, and
Peng~\cite{YanYangPeng} proved a spectral version of the generalized
Alon--Frankl theorem, in which cliques and matchings are restricted
simultaneously.

Nikiforov~\cite{Nikiforov} introduced the matrix
\[
 A_\alpha(G)=\alpha D(G)+(1-\alpha)A(G)
\]
for $0\leq\alpha\leq1$, which combines the adjacency and degree matrices.
Chen and Huang~\cite{ChenHuang} determined the maximum
$A_\alpha$-spectral radius among graphs with a prescribed matching number.
In particular, for $n\geq2t-1$ and $0\leq\alpha<1$, among $n$-vertex graphs
with matching number $t-1$, the extremal graph is
$K_{2t-1}\cup\overline K_{n-2t+1}$ when
$n<(3-\alpha)t-1$, and it is
$K_{t-1}\vee\overline K_{n-t+1}$ when
$n>(3-\alpha)t-1$; at equality, both graphs are extremal.
Using the equivalent normalization $\gamma D(G)+A(G)$, where
$\gamma=\alpha/(1-\alpha)$, Yuan and Shao~\cite{YuanShao} also treated the
equality case. In this matrix problem, the extremal graph changes at
$n=(3-\alpha)t-1$. In the tensor setting introduced below, an explicit admissible
value of $n_0=n_0(t,s,\alpha)$ is given in Remark~\ref{rem:n0} through the roots of auxiliary polynomials.
These results give the $A_\alpha$-matrix version of the Erd\H{o}s--Gallai
problem. This matrix problem is the case $s=2$ of our setting: since every edge
is a $2$-clique, the matrix $A_\alpha(G)$ coincides with the tensor
$\mathcal A_{\alpha,2}(G)$ defined in Section~\ref{sec:prelim}.

Tensor spectral methods extend spectral graph theory to uniform hypergraphs.
Qi~\cite{Qi2005} and Lim~\cite{Lim2005} introduced tensor eigenvalues.
Cooper and Dutle~\cite{CooperDutle} defined the adjacency tensor of a
uniform hypergraph. Liu and Bu~\cite{LiuBu} applied this construction to the
$s$-cliques of a graph to define its $s$-clique tensor. Yu and
Peng~\cite{YuPengZykov} proved the spectral Zykov theorem for the $s$-clique
tensor, the spectral analogue of $\mathrm{ex}(n,K_s,K_{r+1})$.

Since the spectral radius $\rho_s(G)$ of the $s$-clique tensor (defined in
Section~\ref{sec:prelim}) may be viewed as a spectral analogue of the number
of $s$-cliques in $G$, the theorem of Yu and Peng~\cite{YuPeng} is a spectral
counterpart of Wang's generalized Erd\H{o}s--Gallai theorem for
$\mathrm{ex}(n,K_s,M_t)$. To state their result, For \(n\ge 2t-1\) and \(0\le \ell\le t-1\), let \(F^{t,\ell}(n)\) be the graph on \([n]\) such that \([2t-1-\ell]\) induces a clique, \(\{2t-\ell,\ldots,n\}\) is an independent set, and each vertex in \(\{2t-\ell,\ldots,n\}\) is adjacent precisely to the vertices in \([\ell]\). The two cases $\ell=0$ and
$\ell=t-1$ of this family are the graphs that will occur below:
\[
 F^{t,0}(n)=K_{2t-1}\cup\overline K_{n-2t+1},
 \qquad
 F^{t,t-1}(n)=K_{t-1}\vee\overline K_{n-t+1}.
\]
Each of these graphs is
$M_t$-free: if a matching of $F^{t,\ell}(n)$ uses $k$ edges joining
$\{1,\ldots,\ell\}$ to $\{2t-\ell,\ldots,n\}$, then $k\leq\ell$ and those
edges occupy $k$ of the $2t-1-\ell$ clique vertices, while its remaining
edges lie inside the other $2t-1-\ell-k$ of them, so its size is at most
\[
 k+\left\lfloor\frac{2t-1-\ell-k}{2}\right\rfloor
 =\left\lfloor\frac{2t-1-\ell+k}{2}\right\rfloor
 \leq\left\lfloor\frac{2t-1}{2}\right\rfloor=t-1.
\]
Their theorem states that, for fixed $t\geq2$ and $s\geq2$, every $n$-vertex
$M_t$-free graph $G$ with sufficiently large $n$ satisfies
\[
 \rho_s(G)\leq
 \max\left\{\rho_s\bigl(K_{2t-1}\cup\overline K_{n-2t+1}\bigr),
                  \rho_s\bigl(F^{t,t-1}(n)\bigr)\right\}.
\]

For uniform hypergraphs, Lin, Guo, and Zhou~\cite{LinGuoZhou} introduced
the convex combination $\alpha \mathcal{D}+(1-\alpha)\mathcal{A}$ of the degree and adjacency
tensors, and Guo and
Zhou~\cite{GuoZhou} developed the corresponding $\alpha$-spectral theory.
Motivated by the results above on the $A_\alpha$-matrix and the $s$-clique
tensor, we
consider the $\mathcal A_\alpha$-tensor of the $s$-uniform clique hypergraph of
a graph, which combines the $s$-clique tensor with its degree tensor and is
the tensor analogue of the matrix $A_\alpha(G)$.

A direct estimate is obtained from the variational characterization of the
tensor spectral radius (see~\eqref{eq:variational} in
Section~\ref{sec:prelim}) by maximizing the two terms of the variational
formula separately:
\[
 \rhoas(G)\leq\alpha\,\Delta_s(G)+(1-\alpha)\,\rho_s(G),
\]
where $\Delta_s(G)$ is the maximum $s$-clique degree and $\rho_s(G)$ is the
$s$-clique tensor spectral radius. The
bound is sharp when $t<s\leq2t-1$, but generally not when
$3\leq s\leq t$: on $F^{t,t-1}(n)$ the vectors attaining the two separate
maxima need not coincide. The range $3\leq s\leq t$ therefore requires a joint
analysis of the variational formula.

If $s\geq2t$, then the matching number of $K_s$ is
$\lfloor s/2\rfloor\geq t$, so $K_s$ contains $M_t$. Hence every
$M_t$-free graph is $K_s$-free. Consequently, with the notation of
Section~\ref{sec:prelim}, we have $\Cs(G)=\varnothing$,
$\Aas(G)=0$, and $\rhoas(G)=0$ for every $0\leq\alpha\leq1$. This range is
degenerate, and the main theorem below addresses the range
$3\leq s\leq2t-1$.

\begin{theorem}\label{thm:matching}
Let $t\geq2$, $3\leq s\leq2t-1$, $0\leq\alpha\leq1$, and let $G$ be an $n$-vertex
$M_t$-free graph.
\begin{enumerate}[label=\textup{(\roman*)}]
\item If $3\leq s\leq t$, then there exists
$n_0=n_0(t,s,\alpha)$ such that, whenever $n\geq n_0$,
\[
 \rhoas(G)\leq\rhoas(\Fstar).
\]
\item If $t<s\leq2t-1$ and $n\geq2t-1$, then
\[
 \rhoas(G)\leq\rhoas(\Kstar)=\binom{2t-2}{s-1}.
\]
\end{enumerate}
\end{theorem}

A concrete admissible value of $n_0$ for $0<\alpha<1$ is given in
Remark~\ref{rem:n0}. At $\alpha=0$, $n_0$ may be chosen large enough for the
known theorem for the $s$-clique spectral radius and the comparison in the
proof, while at
$\alpha=1$ an admissible $n_0$ follows by comparing the finitely many linear
functions in Proposition~\ref{prop:alpha-one}. Since every $F^{t,\ell}(n)$
is $M_t$-free, the bounds in \textup{(i)} and \textup{(ii)} are attained, by
$\Fstar=K_{t-1}\vee\overline K_{n-t+1}$ and by
$\Kstar$ respectively.
At $\alpha=0$, the problem reduces to the $s$-clique tensor studied by Yu and
Peng; at $\alpha=1$, the tensor is diagonal and its spectral radius equals the
maximum $s$-clique degree $\Delta_s(G)$. Theorem~\ref{thm:matching} includes
both $\alpha=0$ and $\alpha=1$. The following propositions give stronger conclusions. For
$s=t\geq3$,
Proposition~\ref{prop:maxform} shows that the
maximum is $\max\{\rhoas(\Kstar),\rhoas(\Fstar)\}$ for every
$n\geq2t-1$ and every $0\leq\alpha\leq1$. At $\alpha=0$, this removes the
hypothesis that $n$ be sufficiently large
from~\cite[Theorem~1.8]{YuPeng} when $s=t\geq3$
(Remark~\ref{rem:yupeng-improve}). Proposition~\ref{prop:alpha-one} determines
the maximum when $\alpha=1$ for all $s\geq3$ as an explicit maximum of finitely
many linear functions of $n$.

The proof consists of a structural reduction followed by a spectral
comparison. First, we order the vertices by the coordinates of a Perron
vector (a nonnegative eigenvector associated with the spectral radius) of
$\Aas(G)$ and apply graph
compressions. The compressions preserve $M_t$-freeness and do not decrease the
spectral radius. A result of Jiang, Peng, and
Wu~\cite{JiangPengWu} shows that the resulting left-compressed graph is a
subgraph of $F^{t,\ell}(n)$ for some $0\leq\ell\leq t-1$. Second, we describe
the $s$-cliques of $F^{t,\ell}(n)$ and compare the degree and clique terms
in the variational formula. Section~\ref{sec:prelim} gives the definitions
and preliminary facts, and Section~\ref{sec:main} contains the proof of the
main theorem.

\section{Definitions and preliminary facts}\label{sec:prelim}

Throughout this paper, all graphs are finite, simple, and undirected. For a
graph $G$ we write $V(G)$ and $E(G)$ for its vertex and edge sets, and an
\emph{$n$-vertex graph} is one with $n=|V(G)|$; unless stated otherwise the
vertex set is $[n]=\{1,\ldots,n\}$. We use the convention $\binom ab=0$ for
nonnegative integers $a,b$ with $b>a$. For
$s\geq2$, let $\Cs(G)$ be the set
of $s$-cliques of $G$. The $s$-clique degree of a vertex
$v$ is
\[
 d_s^G(v):=|\{C\in\Cs(G):v\in C\}|.
\]
The maximum $s$-clique degree is denoted by
\[
 \Delta_s(G):=\max_{v\in V(G)}d_s^G(v).
\]
For an integer $t\geq1$, $M_t$ denotes a matching of $t$ pairwise disjoint
edges; a graph is $M_t$-free if it contains no $M_t$.
We write $G\vee H$ for the join of $G$ and $H$, and $\overline K_m$ for
the edgeless graph on $m$ vertices.

For an order-$s$, dimension-$n$ complex tensor
\(\mathcal T=(t_{i_1\cdots i_s})\) and \(x\in\mathbb C^n\), write
\[
 \bigl(\mathcal T x^{s-1}\bigr)_i
 =\sum_{i_2,\ldots,i_s=1}^n
 t_{i i_2\cdots i_s}x_{i_2}\cdots x_{i_s}.
\]
Following Qi~\cite{Qi2005} and Lim~\cite{Lim2005}, a number
\(\lambda\in\mathbb C\) is an eigenvalue of
\(\mathcal T\) if there is a nonzero vector \(x\in\mathbb C^n\) such that
\[
 \mathcal T x^{s-1}=\lambda x^{[s-1]},
 \qquad
 x^{[s-1]}=(x_1^{s-1},\ldots,x_n^{s-1})^{\mathsf T}.
\]
The spectral radius is
\[
 \rho(\mathcal T):=\max\{|\lambda|:\lambda\text{ is an eigenvalue of }
 \mathcal T\}.
\]
When \(\mathcal T\) is real, an eigenvalue with a real eigenvector is
called an \(H\)-eigenvalue. For a symmetric nonnegative tensor,
\(\rho(\mathcal T)\) is its largest \(H\)-eigenvalue and admits a
nonnegative eigenvector~\cite{ChangPearsonZhang,Qi2013}.

\begin{definition}[Liu--Bu, Definition~1~\cite{LiuBu}]\label{def:clique-tensor}
The $s$-clique tensor of an $n$-vertex graph $G$ is the order-$s$,
dimension-$n$ symmetric tensor $\As(G)=(a_{i_1\cdots i_s})$ defined by
\[
 a_{i_1\cdots i_s}=
 \begin{cases}
 \dfrac1{(s-1)!},&\{i_1,\ldots,i_s\}\in\Cs(G),\\[2mm]
 0,&\text{otherwise}.
 \end{cases}
\]
\end{definition}

Write
\[
 \rho_s(G):=\rho(\As(G))
\]
for the spectral radius of the $s$-clique tensor.

The tensor $\As(G)$ is the Cooper--Dutle adjacency tensor of the $s$-uniform clique
hypergraph of $G$. Define its degree tensor $\Ds(G)$ by
\[
 (\Ds(G))_{i\cdots i}=d_s^G(i),
 \qquad (\Ds(G))_{i_1\cdots i_s}=0\ \text{otherwise}.
\]

\begin{definition}[$\alpha$-$s$-clique tensor]
For $0\leq\alpha\leq1$, define
\[
 \Aas(G):=\alpha\Ds(G)+(1-\alpha)\As(G),
 \qquad \rhoas(G):=\rho(\Aas(G)).
\]
For $0\leq\alpha<1$, $\Aas(G)$ is the $\mathcal A_\alpha$-tensor of the
$s$-uniform clique hypergraph of $G$ in the terminology of Guo and
Zhou~\cite{GuoZhou}; for completeness we also allow $\alpha=1$. When $\alpha=0$,
$\mathcal A_{0,s}(G)=\As(G)$ and $\rho_{0,s}(G)=\rho_s(G)$, whereas when
$\alpha=1$, $\mathcal A_{1,s}(G)=\Ds(G)$ is diagonal and
$\rho_{1,s}(G)=\Delta_s(G)$.
\end{definition}

Since $\Aas(G)$ is symmetric and nonnegative for every $0\leq\alpha\leq1$,
Qi's variational characterization~\cite[Theorem~2]{Qi2013} yields
\[
\rhoas(G)
=
\max_{\substack{x\in\mathbb{R}_+^n\\ \sum_{i=1}^n x_i^s=1}}
P_{\alpha,s}(G,x),
\]
where
\begin{equation}\label{eq:variational}
P_{\alpha,s}(G,x)
:=
\alpha\sum_{v\in V(G)}d_s^G(v)x_v^s
+
(1-\alpha)s
\sum_{C\in\Cs(G)}
\prod_{v\in C}x_v.
\end{equation}
The two terms in~\eqref{eq:variational} follow from
\[
 x^\top\Ds(G)x^{s-1}=\sum_{v\in V(G)}d_s^G(v)x_v^s
 \quad\text{and}\quad
 x^\top\As(G)x^{s-1}
 =s\sum_{C\in\Cs(G)}\prod_{v\in C}x_v.
\]
The factor $s$ in the second identity comes from the $s!$ orderings of each
clique and the normalization $1/(s-1)!$ in
Definition~\ref{def:clique-tensor}.

We shall use the following standard row-sum bound. Let $\mathcal T$ be a
nonnegative symmetric order-$s$ tensor with row sums $R_i(\mathcal T)$. Choose
a nonnegative eigenvector $x$ for $\rho(\mathcal T)$ and an index $i$ with
$x_i=\max_j x_j>0$. The $i$th eigen-equation gives
\[
 \rho(\mathcal T)x_i^{s-1}
 \leq R_i(\mathcal T)x_i^{s-1},
\]
and hence $\rho(\mathcal T)\leq\max_i R_i(\mathcal T)$. For both $\As(G)$
and $\Aas(G)$, the $v$th row sum is $d_s^G(v)$. Consequently,
\begin{equation}\label{eq:row-sum-bound}
 \rho_s(G)\leq\Delta_s(G),\qquad
 \rhoas(G)\leq\Delta_s(G).
\end{equation}

We record one further immediate consequence of the variational formula, used
repeatedly below. If $G$ and $G'$ are graphs on the same vertex set with
$G\subseteq G'$, then $\Cs(G)\subseteq\Cs(G')$ and $d_s^G(v)\leq d_s^{G'}(v)$
for every $v$, so $P_{\alpha,s}(G,x)\leq P_{\alpha,s}(G',x)$ for every
$x\geq0$; maximizing over $x$ gives
\begin{equation}\label{eq:monotone}
 G\subseteq G'\ \Longrightarrow\ \rhoas(G)\leq\rhoas(G').
\end{equation}

For comparison, as noted in the Introduction, maximizing the two terms of
$P_{\alpha,s}(G,x)$ separately
gives
\begin{equation}\label{eq:convexity}
 \rhoas(G)\leq\alpha\,\Delta_s(G)+(1-\alpha)\,\rho_s(G).
\end{equation}

\section{The generalized Erd\H{o}s--Gallai theorem}\label{sec:main}

This section proves Theorem~\ref{thm:matching} and the two further propositions
announced in the Introduction.

\subsection{Compression and a structural lemma}

For distinct vertices $i,j$, let $\pi_{ij}(G)$ be the standard compression of
$j$ to $i$: for every $u\notin\{i,j\}$ such that $ju\in E(G)$ and
$iu\notin E(G)$, replace $ju$ by $iu$, leaving all other edges unchanged.
We use the following two results of Jiang, Peng, and Wu~\cite{JiangPengWu},
stated here in graph notation.

A graph on $[n]$ is called \emph{left-compressed} if
$\pi_{ij}(G)=G$ for every $1\leq i<j\leq n$. Equivalently, represent each edge
by its ordered pair $(a,b)$ with $a<b$. If $(a,b)$ is an edge and another
ordered pair $(a',b')$ satisfies $a'<b'$, $a'\leq a$, and $b'\leq b$, then
$(a',b')$ is also an edge.

Recall that $M_t$ denotes a matching of $t$ pairwise disjoint edges and that
$F^{t,\ell}(n)$ was defined in the Introduction for $0\leq\ell\leq t-1$.

\begin{lemma}[Jiang--Peng--Wu, Lemma~2.5(a)~\cite{JiangPengWu}]
\label{lem:matching-free}
Let $t\geq2$ be an integer, let $G$ be an $M_t$-free graph on $[n]$, and
let $1\leq i<j\leq n$. Then $\pi_{ij}(G)$ is $M_t$-free.
\end{lemma}

\begin{lemma}[Jiang--Peng--Wu, Lemma~4.3~\cite{JiangPengWu}]
\label{lem:shifted-matching}
Let $n\geq2t$, and let $G$ be a left-compressed $M_t$-free graph on $[n]$.
Then there is an $\ell\in\{0,1,\ldots,t-1\}$ such that
\[
 G\subseteq F^{t,\ell}(n).
\]
\end{lemma}

\begin{proof}
The $t$ edges
$\{r,2t+1-r\}$, $1\leq r\leq t$, are pairwise disjoint, so at least one of
them is absent from $G$. Choose $\ell\in\{0,\ldots,t-1\}$ such that
$\{\ell+1,2t-\ell\}\notin E(G)$. If an edge $ab$, written with $a<b$, satisfied
$a\geq\ell+1$ and $b\geq2t-\ell$, left-compression would force
$\{\ell+1,2t-\ell\}\in E(G)$, a contradiction. Hence every edge either lies
inside $[2t-1-\ell]$ or has an endpoint in $[\ell]$. Therefore
$G\subseteq F^{t,\ell}(n)$.
\end{proof}

\begin{lemma}\label{lem:compression}
Let $s\geq2$, let $G$ be an $n$-vertex graph on $[n]$, let
$0\leq\alpha\leq1$, let $i,j\in[n]$ be distinct, and let
$x\in\mathbb R_+^n$ satisfy $x_i\geq x_j$. Then
\begin{equation}\label{eq:compression-poly}
 P_{\alpha,s}(\pi_{ij}(G),x)\geq P_{\alpha,s}(G,x).
\end{equation}
If, in addition, $\sum_vx_v^s=1$ and $P_{\alpha,s}(G,x)=\rhoas(G)$, then
\[
 \rhoas(\pi_{ij}(G))\geq\rhoas(G).
\]
\end{lemma}

\begin{proof}
Set $H=\pi_{ij}(G)$ and write the variational polynomial as
\[
 P_{\alpha,s}(G,x)=\sum_{C\in\Cs(G)}Q_C(x),\qquad
 Q_C(x):=\alpha\sum_{v\in C}x_v^s+(1-\alpha)s\prod_{v\in C}x_v.
\]
Let $\mathcal L=\Cs(G)\setminus\Cs(H)$ and
$\mathcal N=\Cs(H)\setminus\Cs(G)$. If $C\in\mathcal L$, then $C$ contains at
least one deleted edge $ju$ with $u\notin\{i,j\}$, $ju\in E(G)$, and
$iu\notin E(G)$; in particular $j,u\in C$. Since $C$ is a clique of $G$, the
presence of $i$ in $C$ would force $iu\in E(G)$, a contradiction; hence
$j\in C$ and $i\notin C$. Consider the replacement
$\phi(C)=(C\setminus\{j\})\cup\{i\}$. All edges with both endpoints in
$C\setminus\{j\}$ are unchanged by the compression. Moreover, for every
$v\in C\setminus\{j\}$ we have $jv\in E(G)$, and if $iv\notin E(G)$ then $iv$
is one of the edges created by the compression; in either case
$iv\in E(H)$. Hence $\phi(C)$ is an $s$-clique of $H$. The missing edge $iu$
shows that $\phi(C)$ is not a clique of $G$, so $\phi(C)\in\mathcal N$; and
$\phi$ is injective, since
$i\in\phi(C)$ and $C=(\phi(C)\setminus\{i\})\cup\{j\}$. By the definitions of
$\mathcal L$ and $\mathcal N$,
$\Cs(H)=(\Cs(G)\setminus\mathcal L)\,\dot\cup\,\mathcal N$, and every clique
in $\mathcal N\setminus\phi(\mathcal L)$ contributes a nonnegative term.

For $C'=\phi(C)$, the difference between the corresponding terms is
\begin{align*}
 Q_{C'}(x)-Q_C(x)
 &=\alpha\left(\sum_{v\in C'}x_v^s-\sum_{v\in C}x_v^s\right)
 +(1-\alpha)s\left(\prod_{v\in C'}x_v-\prod_{v\in C}x_v\right)\\
 &=\alpha(x_i^s-x_j^s)+(1-\alpha)s(x_i-x_j)
 \prod_{v\in C\setminus\{j\}}x_v\geq0.
\end{align*}
Summing these inequalities and adding the remaining terms
gives~\eqref{eq:compression-poly}. If, moreover, $x$ is an optimal solution of the maximization
problem in~\eqref{eq:variational} for $G$, then
\[
 \rhoas(H)\geq P_{\alpha,s}(H,x)
 \geq P_{\alpha,s}(G,x)=\rhoas(G),
\]
as required.
\end{proof}

\subsection{Proof of Theorem~\ref{thm:matching}}

\begin{proof}
We first prove part~\textup{(i)} at $\alpha=0$. By the theorem of Yu and
Peng~\cite[Theorem~1.8]{YuPeng}, for fixed $t$ and $s$ and all sufficiently
large $n$, every $n$-vertex $M_t$-free graph satisfies
\[
 \rho_s(G)\leq\max\{\rho_s(\Kstar),\rho_s(\Fstar)\}.
\]
It remains only to identify which term is larger for large $n$. Put
$m=n-t+1$, fix an $(s-1)$-subset $A$ of the clique part of
$\Fstar=K_{t-1}\vee\overline K_m$, and define a nonnegative vector by
\[
 x_u=s^{-1/s}\quad(u\in A),
 \qquad
 x_w=(sm)^{-1/s}\quad(w\in\overline K_m),
 \qquad x_v=0\quad\text{otherwise}.
\]
Then $\sum_vx_v^s=(s-1)/s+1/s=1$, and every
$A\cup\{w\}$ with $w\in\overline K_m$ is an $s$-clique. The variational
formula therefore gives
\[
 \rho_s(\Fstar)\geq
 s m(s^{-1/s})^{s-1}(sm)^{-1/s}=m^{1-1/s}.
\]
The right-hand side tends to infinity with $n$, whereas
$\rho_s(\Kstar)=\binom{2t-2}{s-1}$ is constant. Thus
$\rho_s(\Fstar)>\rho_s(\Kstar)$ for all sufficiently large $n$, proving
part~\textup{(i)} when $\alpha=0$.

If $t<s\leq2t-1$ and $n=2t-1$, every $n$-vertex graph is $M_t$-free, and
$G\subseteq K_{2t-1}$. By~\eqref{eq:monotone},
\[
 \rhoas(G)\leq\rhoas(K_{2t-1})=\binom{2t-2}{s-1}.
\]
Thus part~\textup{(ii)} holds when $n=2t-1$. For the remainder of the proof,
assume $n\geq2t$.
Choose an $n$-vertex $M_t$-free graph $G^*$ with maximum $\rhoas$, and let $x$
be a normalized nonnegative Perron vector of $\Aas(G^*)$. Relabel the vertices so
that $x_1\geq\cdots\geq x_n$, and set $G_0=G^*$. As long as the
current graph $G_k$ is not left-compressed, choose $i<j$ such that
$\pi_{ij}(G_k)\ne G_k$ and set $G_{k+1}=\pi_{ij}(G_k)$. Compression preserves
$M_t$-freeness by Lemma~\ref{lem:matching-free}. We now prove simultaneously
that every step preserves the extremal spectral radius and that the fixed
vector $x$ is still an optimal solution of the maximization problem in~\eqref{eq:variational}. We show by induction that
\[
 \rhoas(G_k)=\rhoas(G^*)
 \quad\text{and}\quad
 P_{\alpha,s}(G_k,x)=\rhoas(G_k).
\]
These equalities hold for $k=0$. By the choice of $G^*$, at the $k$th step we
have
\[
 \rhoas(G^*)\geq\rhoas(G_{k+1})
 \geq P_{\alpha,s}(G_{k+1},x)
 \geq P_{\alpha,s}(G_k,x)=\rhoas(G_k),
\]
where the last inequality is~\eqref{eq:compression-poly}, applicable because
$i<j$ and the coordinates of $x$ are nonincreasing.
The first and last terms are equal, so equality holds throughout. In
particular, $P_{\alpha,s}(G_{k+1},x)=\rhoas(G_{k+1})$. Thus the same vector $x$
and the same ordering of its coordinates can be used at every step. If a
compression moves $q_k\geq1$ edges, then the integer-valued potential
$\sum_{uv\in E(G_k)}(u+v)$ decreases by exactly $q_k(j-i)>0$. Hence after
finitely many steps we obtain a left-compressed extremal graph $H$.

By Lemma~\ref{lem:shifted-matching}, $H\subseteq F^{t,\ell}(n)$ for some
$0\leq\ell\leq t-1$. The variational formula gives
\[
 \rhoas(G^*)=\rhoas(H)\leq\rhoas(F^{t,\ell}(n)).
\]
It remains to compare the spectral radii of the graphs $F^{t,\ell}(n)$. Fix
$0\leq\ell\leq t-1$, and write
\[
 p_\ell=2t-1-\ell,\qquad N_\ell=n-p_\ell,\qquad
 S=[\ell],\quad B=[p_\ell],\quad T=\{p_\ell+1,\ldots,n\}.
\]
The set $B$ induces a clique, the set $T$ is independent, and every vertex of
$T$ is adjacent exactly to the vertices in $S$. Every $s$-clique meeting $T$
contains exactly one vertex of $T$ and $s-1$ vertices of $S$. We therefore
define
\[
 \mathcal K_\ell=\binom{B}{s},\qquad
 \mathcal T_\ell=
 \left\{A\cup\{w\}:A\in\binom{S}{s-1},\ w\in T\right\}.
\]
Thus $\mathcal K_\ell$ is the set of $s$-cliques contained
in $B$, and $\mathcal T_\ell$ is the set of $s$-cliques that meet $T$. Hence
\begin{equation}\label{eq:clique-decomposition}
\Cs(F^{t,\ell}(n))=\mathcal K_\ell\,\dot\cup\,\mathcal T_\ell.
\end{equation}

If $\ell\leq s-2$, then $|S|<s-1$ and $\mathcal T_\ell=\varnothing$. Hence
every $s$-clique is contained in $B$, and every vertex in $T$ has $s$-clique
degree zero, that is, the rows of $\Aas(F^{t,\ell}(n))$ indexed by $T$ are
zero rows although these vertices are not isolated in the graph. It follows
that the variational polynomial is independent of the coordinates $x_j$,
$j\in T$, and

\begin{equation}\label{eq:small-core-bound}
     \rhoas(F^{t,\ell}(n))=\rhoas(K_{p_\ell})
 \leq\rhoas(K_{2t-1}).
\end{equation}
For every positive integer $p$,
\(\rho_s(K_p)=\rhoas(K_p)=\binom{p-1}{s-1}\): the $s$-clique hypergraph is
regular with this degree, so both $\As(K_p)$ and $\Aas(K_p)$ have the same
constant row sum, and the claim follows from the all-ones vector
and~\eqref{eq:row-sum-bound}.
The same formula remains valid for $p<s$, since both sides are zero.

Assume $s-1\leq\ell\leq t-1$. The $s$-clique degrees of vertices in
$S$, $B\setminus S$, and $T$, respectively, are
\[
 N_\ell\binom{\ell-1}{s-2}+\binom{p_\ell-1}{s-1},
 \qquad \binom{p_\ell-1}{s-1},
 \qquad \binom\ell{s-1}.
\]
The degree of a vertex in $S$ is at least that of a vertex in $B\setminus S$.
To compare the degrees in $S$ and $T$, note that $n\geq2t$ gives
$N_\ell=n-2t+1+\ell\geq\ell+1$, while
\[
 \ell\binom{\ell-1}{s-2}=(s-1)\binom\ell{s-1}
 \geq\binom\ell{s-1}.
\]
Therefore
$N_\ell\binom{\ell-1}{s-2}\geq\binom\ell{s-1}$, so a vertex in $S$ also has
degree at least that of a vertex in $T$. Consequently,
\[
 \Delta_s(F^{t,\ell}(n))=na_\ell+c_\ell,
\]
where
\[
 a_\ell=\binom{\ell-1}{s-2},\qquad
 c_\ell=\binom{2t-2-\ell}{s-1}-(2t-1-\ell)a_\ell.
\]

Let $x\geq0$ satisfy $\sum_v x_v^s=1$, and write
\[
 X=\sum_{i\in S}x_i^s,\qquad Y=\sum_{j\in T}x_j^s,
 \qquad e_{s-1}(x_S)=\sum_{A\in\binom S{s-1}}\prod_{i\in A}x_i.
\]
By~\eqref{eq:clique-decomposition}, the sum over $\mathcal T_\ell$ factors as
$e_{s-1}(x_S)\sum_{j\in T}x_j$. Maclaurin's inequality for elementary
symmetric means, combined with the power-mean inequality, gives the first
estimate below; H\"older's inequality gives the second; and the weighted
AM--GM inequality, applied to $X/\lambda$ and $Y/\mu$ with weights
$\lambda=(s-1)/s$ and $\mu=1/s$, gives the
third, where $X+Y\leq1$ holds because $S\cap T=\varnothing$:
\begin{align*}
 e_{s-1}(x_S)&\leq
 \binom\ell{s-1}\ell^{-(s-1)/s}X^{(s-1)/s},\\
 \sum_{j\in T}x_j&\leq N_\ell^{1-1/s}Y^{1/s},\\
 X^{(s-1)/s}Y^{1/s}&\leq
 \left(\frac{s-1}{s}\right)^{(s-1)/s}s^{-1/s}.
\end{align*}
Hence
\[
 s\sum_{C\in\mathcal T_\ell}\prod_{v\in C}x_v
 \leq D_\ell N_\ell^{1-1/s},
 \qquad
 D_\ell=s\binom{\ell}{s-1}\ell^{-(s-1)/s}
 \left(\frac{s-1}{s}\right)^{(s-1)/s}s^{-1/s}.
\]
Put $W=\sum_{i\in B}x_i^s$. If $W>0$, apply the variational formula for the
$s$-clique tensor to the normalized vector $W^{-1/s}x_B$. Rescaling gives the
first inequality below. If $W=0$, that inequality is immediate. Thus
\[
 s\sum_{C\in\mathcal K_\ell}\prod_{v\in C}x_v
 \leq\rho_s(K_{p_\ell})\sum_{i\in B}x_i^s
 \leq\binom{p_\ell-1}{s-1}.
\]
Here we used the formula
$\rho_s(K_{p_\ell})=\binom{p_\ell-1}{s-1}$ established above and $W\leq1$.
The degree contribution is at most
\(\Delta_s(F^{t,\ell}(n))\sum_v x_v^s=na_\ell+c_\ell\).
The variational formula yields
\begin{equation}\label{eq:upperF}
 \rhoas(F^{t,\ell}(n))
 \leq\alpha(na_\ell+c_\ell)+(1-\alpha)
 \left\{\binom{p_\ell-1}{s-1}+D_\ell N_\ell^{1-1/s}\right\}.
\end{equation}
Choose a vertex $u\in[t-1]$ in the clique part of $\Fstar$ and substitute the
standard basis vector $e_u$ into the variational formula. The clique term vanishes,
and the degree term equals $\alpha d_s^{\Fstar}(u)$. Substituting $\ell=t-1$
into the degree formula gives
\begin{equation}\label{eq:lowerF}
 \rhoas(\Fstar)\geq\alpha(na_{t-1}+c_{t-1}).
\end{equation}

Suppose $3\leq s<t$ and $\alpha>0$. For $s-1\leq\ell\leq t-2$, let
\[
 A_\ell=\alpha(a_{t-1}-a_\ell),\qquad
 C_\ell=\alpha\max\{c_\ell-c_{t-1},0\}
 +(1-\alpha)\binom{p_\ell-1}{s-1}.
\]
Pascal's identity gives
\[
 a_{t-1}-a_\ell
 =\binom{t-2}{s-2}-\binom{\ell-1}{s-2}
 =\sum_{k=\ell}^{t-2}\binom{k-1}{s-3}>0,
\]
so $A_\ell>0$. From~\eqref{eq:upperF} and~\eqref{eq:lowerF}, using
$N_\ell\leq n$, we obtain
\begin{align*}
 &\rhoas(\Fstar)-\rhoas(F^{t,\ell}(n))\\
 &\quad\geq\alpha\{n(a_{t-1}-a_\ell)+c_{t-1}-c_\ell\}
 -(1-\alpha)\binom{p_\ell-1}{s-1}
 -(1-\alpha)D_\ell N_\ell^{1-1/s}\\
 &\quad\geq A_\ell n-C_\ell-(1-\alpha)D_\ell n^{1-1/s}.
\end{align*}
Because $1-1/s<1$ and $A_\ell>0$, the positive linear term in the last bound
dominates the remaining terms. Hence
$\rhoas(\Fstar)>\rhoas(F^{t,\ell}(n))$ for all sufficiently large $n$. When
$3\leq s<t$, inequality~\eqref{eq:small-core-bound} applies to
$0\leq\ell\leq s-2$, the preceding argument gives
$\rhoas(\Fstar)>\rhoas(F^{t,\ell}(n))$ for $s-1\leq\ell\leq t-2$, and the
remaining graph is $\Fstar$.

We now finish the proof of part~\textup{(i)} for $\alpha>0$. If
$3\leq s\leq t$, then
$a_{t-1}=\binom{t-2}{s-2}\geq1$. By~\eqref{eq:lowerF},
$\rhoas(\Fstar)$ grows at least linearly with $n$, whereas
\[
\rhoas(\Kstar)=\rhoas(K_{2t-1})=\binom{2t-2}{s-1}
\]
is independent of $n$. Hence
$\rhoas(\Fstar)>\rhoas(\Kstar)$ for all sufficiently
large $n$. For $3\leq s<t$, this inequality,~\eqref{eq:small-core-bound}, and
the comparisons for $s-1\leq\ell\leq t-2$ show that
$\rhoas(\Fstar)\geq\rhoas(F^{t,\ell}(n))$ for every $\ell$. For $s=t\geq3$, every
graph $F^{t,\ell}(n)$ with
$0\leq\ell\leq t-2$ satisfies~\eqref{eq:small-core-bound}, and the remaining
graph is $\Fstar$. In either case,
$\rhoas(G^*)\leq\rhoas(\Fstar)$. Since $G^*$ has maximum spectral radius and
$\Fstar$ is $M_t$-free, this proves the sharp bound in part~\textup{(i)} for
$\alpha>0$. The case $\alpha=0$ was proved at the beginning.

Finally, we prove part~\textup{(ii)} for every $0\leq\alpha\leq1$ and
$n\geq2t$. If $t<s\leq2t-1$, then
$t-1\leq s-2$, so every graph
$F^{t,\ell}(n)$ satisfies~\eqref{eq:small-core-bound}. Consequently,
$\rhoas(G^*)\leq\rhoas(\Kstar)$ for every $n\geq2t$. Since $G^*$ has maximum
spectral radius and $\Kstar$ is $M_t$-free, this proves the sharp bound in
part~\textup{(ii)}.
\end{proof}

The preceding reduction by compression is valid for every
$0\leq\alpha\leq1$. Indeed, let
$G$ be any $n$-vertex $M_t$-free graph with $n\geq2t$, choose a nonnegative
vector $x$ which is an optimal solution of the maximization
problem in~\eqref{eq:variational}, and relabel the vertices so that
$x_1\geq\cdots\geq x_n$. Repeated application
of~\eqref{eq:compression-poly} with this fixed vector yields
\[
 P_{\alpha,s}(H,x)\geq P_{\alpha,s}(G,x)=\rhoas(G)
\]
for the terminal left-compressed graph $H$, and hence
$\rhoas(G)\leq\rhoas(H)$. Lemma~\ref{lem:matching-free} preserves
$M_t$-freeness, the decreasing potential guarantees termination,
Lemma~\ref{lem:shifted-matching} gives
$H\subseteq F^{t,\ell}(n)$ for some $0\leq\ell\leq t-1$, and
monotonicity~\eqref{eq:monotone} yields
\[
 \rhoas(G)\leq\rhoas(F^{t,\ell}(n)).
\]
No step in this reduction requires $\alpha$ to lie in the interior of
$[0,1]$. The assumption
$\alpha>0$ is used only in the subsequent
asymptotic comparison, through $A_\ell>0$ and~\eqref{eq:lowerF}.

When $s=t\geq3$, every $0\leq\ell\leq t-2$ satisfies $\ell\leq s-2$, so
inequality~\eqref{eq:small-core-bound} gives
$\rhoas(F^{t,\ell}(n))\leq\rhoas(\Kstar)$. This yields the following bound for
$n\geq2t-1$ and $0\leq\alpha\leq1$.

\begin{proposition}\label{prop:maxform}
Let $s=t\geq3$, $0\leq\alpha\leq1$, and let $G$ be an $n$-vertex
$M_t$-free graph with $n\geq2t-1$. Then
\[
 \rhoas(G)\leq\max\bigl\{\rhoas(\Kstar),\rhoas(\Fstar)\bigr\}
 =\max\left\{\tbinom{2t-2}{t-1},\,\rhoas(\Fstar)\right\}.
\]
Both $\Kstar$ and $\Fstar$ are $M_t$-free, so this maximum is attained.
\end{proposition}

\begin{proof}
For $n=2t-1$, monotonicity~\eqref{eq:monotone} gives
$\rhoas(G)\leq\rhoas(\Kstar)$. For $n\geq2t$, the preceding reduction gives
\[
 \rhoas(G)\leq
 \max_{0\leq\ell\leq t-1}\rhoas(F^{t,\ell}(n)).
\]
Since $s=t\geq3$, inequality~\eqref{eq:small-core-bound} yields
$\rhoas(F^{t,\ell}(n))\leq\rhoas(\Kstar)$ for every $\ell\leq t-2$. For the
remaining value $\ell=t-1$, the corresponding term is $\rhoas(\Fstar)$. The
result follows.
\end{proof}

Which of the two graphs attains the maximum in
Proposition~\ref{prop:maxform} is decided by comparing the constant
$\binom{2t-2}{t-1}$ with $\rhoas(\Fstar)$. For $0\leq\alpha<1$, the latter can
be computed from the two-coordinate eigenvalue equations in
Remark~\ref{rem:perron}; at
$\alpha=1$, it equals $\Delta_s(\Fstar)$, as described in
Proposition~\ref{prop:alpha-one}.

\begin{remark}\label{rem:yupeng-improve}
When $\alpha=0$, we have $\rho_{0,s}=\rho_s$, so
Proposition~\ref{prop:maxform} reads
\[
 \rho_s(G)\leq\max\bigl\{\rho_s(\Kstar),\rho_s(\Fstar)\bigr\}
 \qquad(s=t\geq3,\ n\geq2t-1).
\]
Thus, for $s=t\geq3$, Proposition~\ref{prop:maxform} removes the hypothesis
in~\cite[Theorem~1.8]{YuPeng} that $n$ be sufficiently large. For $3\leq s<t$,
it appears to remain open whether the same inequality holds for every
$n\geq2t-1$ at $\alpha=0$ (see Problem~\ref{problem:exact-range}).
\end{remark}

When $\alpha=1$, the tensor $\Aas(G)$ is diagonal and
$\rho_{1,s}(G)=\Delta_s(G)$. Applying the same reduction with $\alpha=1$
therefore reduces
the problem to maximizing the $s$-clique degree and gives the following
explicit maximum of finitely many linear functions of $n$.

\begin{proposition}\label{prop:alpha-one}
Let $t\geq2$, $s\geq3$, and let $G$ be an $n$-vertex $M_t$-free graph with
$n\geq2t-1$. Then
\[
 \Delta_s(G)\leq
 \max\left\{\binom{2t-2}{s-1},\
 \max_{s-1\leq\ell\leq t-1}\bigl(na_\ell+c_\ell\bigr)\right\},
\]
where $a_\ell=\binom{\ell-1}{s-2}$ and
$c_\ell=\binom{2t-2-\ell}{s-1}-(2t-1-\ell)a_\ell$. The second maximum is over
an empty index set, and so is omitted, when $s>t$; the bound is attained by one
of the graphs $F^{t,\ell}(n)$.
\end{proposition}

\begin{proof}
For $n=2t-1$, monotonicity~\eqref{eq:monotone} gives
$\Delta_s(G)\leq\binom{2t-2}{s-1}$. For $s-1\leq\ell\leq t-1$, the degree
formula remains valid when $n=2t-1$: indeed, $N_\ell=\ell$ and
$\ell a_\ell=(s-1)\binom{\ell}{s-1}\geq\binom{\ell}{s-1}$, so a vertex in
$S$ still has maximum $s$-clique degree. Since
$F^{t,\ell}(2t-1)\subseteq K_{2t-1}$, we have
$(2t-1)a_\ell+c_\ell\leq\binom{2t-2}{s-1}$. Thus the displayed maximum equals
its first term when $n=2t-1$ and is attained by
$F^{t,0}(2t-1)=K_{2t-1}$. We may therefore assume $n\geq2t$. Compression and
Lemma~\ref{lem:shifted-matching} give
\[
 \Delta_s(G)\leq
 \max_{0\leq\ell\leq t-1}\Delta_s(F^{t,\ell}(n)).
\]
For $\ell\leq s-2$, inequality~\eqref{eq:small-core-bound} gives
$\Delta_s(F^{t,\ell}(n))\leq\binom{2t-2}{s-1}$, whereas for
$s-1\leq\ell\leq t-1$ the degree formula gives
$\Delta_s(F^{t,\ell}(n))=na_\ell+c_\ell$. Taking the maximum proves the
bound.
\end{proof}

\begin{remark}\label{rem:n0}
Assume $0<\alpha<1$. The proof gives the following explicit choice of $n_0$
when $3\leq s<t$.
For each $\ell\in\{s-1,\ldots,t-2\}$, let $y_\ell$ be the unique positive
root of
\[
 A_\ell y^s-(1-\alpha)D_\ell y^{s-1}-C_\ell=0.
\]
Then one may take
\[
 n_0=\max\left\{2t,\
 1+\left\lfloor
 \frac{\binom{2t-2}{s-1}/\alpha-c_{t-1}}{a_{t-1}}
 \right\rfloor,\
 1+\left\lfloor
 \max_{s-1\leq\ell\leq t-2}y_\ell^s\right\rfloor\right\}.
\]
Thus $n_0$ is obtained by solving finitely many equations in one variable. Put
$f_\ell(y)=A_\ell y^s-(1-\alpha)D_\ell y^{s-1}-C_\ell$. Its coefficients have
at most one sign change, so Descartes' rule gives at most one positive root. If
$C_\ell>0$, then $f_\ell(0)<0$ and $f_\ell(y)\to\infty$, proving existence by
continuity; if $C_\ell=0$, then
$f_\ell(y)=y^{s-1}(A_\ell y-(1-\alpha)D_\ell)$ and again has exactly one
positive root. Thus $y_\ell$ exists and is unique. Put
$B_\ell=(1-\alpha)D_\ell>0$. At the positive root $y_\ell$,
\[
 A_\ell y_\ell-B_\ell=\frac{C_\ell}{y_\ell^{s-1}}\geq0.
\]
Hence, for every $y\geq y_\ell$,
\[
 f_\ell'(y)
 =y^{s-2}\bigl(sA_\ell y-(s-1)B_\ell\bigr)
 \geq y^{s-2}\left(B_\ell+
 \frac{sC_\ell}{y_\ell^{s-1}}\right)>0.
\]
Thus $f_\ell$ is strictly increasing on $[y_\ell,\infty)$. The lower bound for
the spectral-radius difference is $f_\ell(n^{1/s})$, and
$n\geq1+\lfloor y_\ell^s\rfloor$ implies $n^{1/s}>y_\ell$. Hence this lower
bound is strictly positive for $n\geq n_0$, that is,
$\rhoas(\Fstar)-\rhoas(F^{t,\ell}(n))>0$ for every
$s-1\leq\ell\leq t-2$. For $s=t\geq3$, still assuming $0<\alpha<1$, the same
comparison with the constant $\binom{2t-2}{t-1}$ shows that one may take
\[
 n_0=\max\left\{2t,\
 1+\left\lfloor
 \frac{\binom{2t-2}{t-1}/\alpha-c_{t-1}}{a_{t-1}}
 \right\rfloor\right\}.
\]
This bound degenerates as $\alpha\to0^+$. For sufficiently large $n$,
the case $\alpha=0$ in part~\textup{(i)} is instead covered by the result
of Yu and Peng~\cite{YuPeng} recalled in the Introduction. The second term in
the displayed maximum defining $n_0$ cannot be omitted: it enforces the
comparison with $K_{2t-1}$, and as $\alpha\to1^-$ the quantity
inside the floor tends to
$\bigl(\binom{2t-2}{s-1}-c_{t-1}\bigr)/a_{t-1}$. This is the value of $n$ for
which $\Delta_s(\Fstar)=na_{t-1}+c_{t-1}$ equals
$\Delta_s(\Kstar)=\binom{2t-2}{s-1}$, i.e. the point at which the two
graphs are both extremal in the case $\alpha=1$.
\end{remark}

\begin{problem}\label{problem:exact-range}
For $0\leq\alpha<1$, determine the least integer $n_0(t,s,\alpha)$ for which
Theorem~\ref{thm:matching}\textup{(i)} holds. More generally, for the same
range of $\alpha$, it is natural
to ask whether
\begin{equation}\label{eq:maxform}
 \rhoas(G)\leq\max\{\rhoas(\Kstar),\rhoas(\Fstar)\}
\end{equation}
holds for \emph{every} $n\geq2t-1$ when $3\leq s\leq t$.
Proposition~\ref{prop:maxform} settles the case $s=t\geq3$, so only
$3\leq s<t$ remains open. In this range, the proof of
Theorem~\ref{thm:matching} already gives
$\rhoas(G)\leq\max_{0\leq\ell\leq t-1}\rhoas(F^{t,\ell}(n))$ for every
$n\geq2t$, and~\eqref{eq:small-core-bound} handles every $\ell\leq s-2$.
Since every $F^{t,\ell}(n)$ is itself $M_t$-free,
inequality~\eqref{eq:maxform} is equivalent to
\[
 \rhoas(F^{t,\ell}(n))\leq\max\{\rhoas(\Kstar),\rhoas(\Fstar)\}
\]
for $s-1\leq\ell\leq t-2$ and $2t\leq n<n_0(t,s,\alpha)$.
Thus it remains only to compare the finitely many pairs $(n,\ell)$ satisfying
these two ranges. When $\alpha=0$ and $n$ is sufficiently large,
inequality~\eqref{eq:maxform} is precisely the conclusion
of~\cite[Theorem~1.8]{YuPeng}; it appears to remain open whether the inequality
holds for every
$n\geq2t-1$.
\end{problem}

\begin{remark}\label{rem:perron}
Assume $3\leq s\leq t$, $0\leq\alpha<1$, and $n\geq2t-1$. Let $C$ and $I$
denote the clique and independent parts of $\Fstar$. For every $u\in C$ and
$w\in I$, one can choose $s-2$ vertices from $C\setminus\{u\}$, since
$t-2\geq s-2$; together with $u$ and $w$, they form an $s$-clique. Hence the
2-section of the $s$-uniform clique hypergraph contains all edges between
$C$ and $I$ and is connected. Since $1-\alpha>0$, $\Aas(\Fstar)$ and
$\As(\Fstar)$ have the same off-diagonal support, so $\Aas(\Fstar)$ is weakly
irreducible. Its normalized positive
Perron vector is therefore unique~\cite{Qi2013}. Automorphisms permuting
vertices within either part leave the tensor invariant, so uniqueness implies
that this vector has a common coordinate $a$ in the clique part and a common
coordinate $b$ in the independent part. Let $r=t-1$ and
$m=n-t+1$, so $ra^s+mb^s=1$. Writing
\[
 d_C=\binom{t-2}{s-1}+m\binom{t-2}{s-2},
 \qquad d_I=\binom{t-1}{s-1},
\]
the value $\rho=\rhoas(\Fstar)$ is determined by
\begin{align*}
 \rho a^{s-1}
 &=\alpha d_C a^{s-1}+(1-\alpha)\left\{
   \binom{t-2}{s-1}a^{s-1}
   +m\binom{t-2}{s-2}a^{s-2}b\right\},\\
 \rho b^{s-1}
 &=\alpha d_I b^{s-1}+(1-\alpha)
   \binom{t-1}{s-1}a^{s-1}.
\end{align*}
The case $\alpha=1$ is excluded because there $\Aas(\Fstar)=\Ds(\Fstar)$
is diagonal and not weakly irreducible. Moreover,
\[
 d_C-d_I=(m-1)\binom{t-2}{s-2}>0,
\]
so the optimal solutions of the maximization problem in~\eqref{eq:variational} for
$\Fstar$ are precisely the nonnegative vectors supported on the clique part. Thus uniqueness and
positivity on all vertices fail.
\end{remark}

\section{Concluding remarks}\label{sec:conclusion}

For every $0\leq\alpha\leq1$, we proved an $\alpha$-$s$-clique spectral
analogue of the generalized
Erd\H{o}s--Gallai theorem by combining graph compression with the variational
characterization of the tensor spectral radius in the range
$3\leq s\leq2t-1$. For $3\leq s\leq t$ and
sufficiently large $n$, the bound is attained by $\Fstar$; for
$t<s\leq2t-1$ and $n\geq2t-1$, it is attained by $\Kstar$. When
$s=t\geq3$,
Proposition~\ref{prop:maxform} determines the maximum for every $n\geq2t-1$
and $0\leq\alpha\leq1$. In particular, at $\alpha=0$ this removes the
requirement that $n$ be sufficiently large from~\cite[Theorem~1.8]{YuPeng}.
Proposition~\ref{prop:alpha-one} also determines the maximum when $\alpha=1$
for all $s\geq3$.

The main remaining problem is the range $3\leq s<t$ and
$0\leq\alpha<1$: determine the least integer $n_0(t,s,\alpha)$ for which
Theorem~\ref{thm:matching}\textup{(i)} holds and decide
whether~\eqref{eq:maxform} holds for
every $n\geq2t-1$; see Problem~\ref{problem:exact-range}. A complete
characterization of all $M_t$-free graphs attaining the bounds in
Theorem~\ref{thm:matching} and Propositions~\ref{prop:maxform}
and~\ref{prop:alpha-one} also remains open.

\section*{Declaration of competing interest}

The authors declare that they have no known competing financial interests or
personal relationships that could have appeared to influence the work
reported in this paper.

\section*{Declaration of generative AI and AI-assisted technologies
in the manuscript preparation process}

During the preparation of this work, the author(s) used GLM-5.3 and DeepSeek 4
for grammar checking, language polishing, computational verification of
combinatorial identities and numerical values, and validation of preliminary
research ideas. The author(s) reviewed and edited the output as needed and
take full responsibility for the content of the published article.

\begingroup
\small
\setlength{\bibsep}{1pt}
\bibliographystyle{abbrvnat} 
\bibliography{alpha_clique_tensor_v8}

\begin{thebibliography}{24}
\providecommand{\natexlab}[1]{#1}
\providecommand{\url}[1]{\texttt{#1}}
\expandafter\ifx\csname urlstyle\endcsname\relax
  \providecommand{\doi}[1]{doi: #1}\else
  \providecommand{\doi}{doi: \begingroup \urlstyle{rm}\Url}\fi

\bibitem[Alon and Shikhelman(2016)]{AlonShikhelman}
N.~Alon and C.~Shikhelman.
\newblock Many {$T$} copies in {$H$}-free graphs.
\newblock \emph{J. Combin. Theory Ser. B}, 121:\penalty0 146--172, 2016.
\newblock \doi{10.1016/j.jctb.2016.03.004}.

\bibitem[Chang et~al.(2008)Chang, Pearson, and Zhang]{ChangPearsonZhang}
K.~C. Chang, K.~Pearson, and T.~Zhang.
\newblock Perron-{F}robenius theorem for nonnegative tensors.
\newblock \emph{Commun. Math. Sci.}, 6\penalty0 (2):\penalty0 507--520, 2008.
\newblock \doi{10.4310/cms.2008.v6.n2.a12}.

\bibitem[Chen and Huang(2022)]{ChenHuang}
Q.~Chen and Q.~Huang.
\newblock The maximal {$A_{\alpha}$}-spectral radius of graphs with given
  matching number.
\newblock \emph{Linear Multilinear Algebra}, 70\penalty0 (20):\penalty0
  5193--5206, 2022.
\newblock \doi{10.1080/03081087.2021.1910120}.

\bibitem[Cooper and Dutle(2012)]{CooperDutle}
J.~Cooper and A.~Dutle.
\newblock Spectra of uniform hypergraphs.
\newblock \emph{Linear Algebra Appl.}, 436\penalty0 (9):\penalty0 3268--3292,
  2012.
\newblock \doi{10.1016/j.laa.2011.11.018}.

\bibitem[Erd{\H{o}}s(1962)]{ErdosCliques}
P.~Erd{\H{o}}s.
\newblock On the number of complete subgraphs contained in certain graphs.
\newblock \emph{Magyar Tud. Akad. Mat. Kutat\'o{} Int. K\"ozl.}, 7:\penalty0
  459--464, 1962.

\bibitem[Erd{\H{o}}s and Gallai(1959)]{ErdosGallai}
P.~Erd{\H{o}}s and T.~Gallai.
\newblock On maximal paths and circuits of graphs.
\newblock \emph{Acta Math. Acad. Sci. Hungar.}, 10:\penalty0 337--356, 1959.
\newblock \doi{10.1007/BF02024498}.

\bibitem[Feng et~al.(2007)Feng, Yu, and Zhang]{FengYuZhang}
L.~Feng, G.~Yu, and X.-D. Zhang.
\newblock Spectral radius of graphs with given matching number.
\newblock \emph{Linear Algebra Appl.}, 422\penalty0 (1):\penalty0 133--138,
  2007.
\newblock \doi{10.1016/j.laa.2006.09.014}.

\bibitem[Guo and Zhou(2020)]{GuoZhou}
H.~Guo and B.~Zhou.
\newblock On the {$\alpha$}-spectral radius of uniform hypergraphs.
\newblock \emph{Discuss. Math. Graph Theory}, 40\penalty0 (2):\penalty0
  559--575, 2020.
\newblock \doi{10.7151/dmgt.2268}.

\bibitem[Jiang et~al.(2018)Jiang, Peng, and Wu]{JiangPengWu}
T.~Jiang, Y.~Peng, and B.~Wu.
\newblock Lagrangian densities of some sparse hypergraphs and {T}ur\'an numbers
  of their extensions.
\newblock \emph{European J. Combin.}, 73:\penalty0 20--36, 2018.
\newblock \doi{10.1016/j.ejc.2018.05.001}.

\bibitem[Lim(2005)]{Lim2005}
L.-H. Lim.
\newblock Singular values and eigenvalues of tensors: a variational approach.
\newblock In \emph{1st IEEE International Workshop on Computational Advances in
  Multi-Sensor Adaptive Processing, 2005.}, pages 129--132, 2005.
\newblock \doi{10.1109/CAMAP.2005.1574201}.

\bibitem[Lin et~al.(2020)Lin, Guo, and Zhou]{LinGuoZhou}
H.~Lin, H.~Guo, and B.~Zhou.
\newblock On the {$\alpha$}-spectral radius of irregular uniform hypergraphs.
\newblock \emph{Linear Multilinear Algebra}, 68\penalty0 (2):\penalty0
  265--277, 2020.
\newblock \doi{10.1080/03081087.2018.1502253}.

\bibitem[Liu and Bu(2023)]{LiuBu}
C.~Liu and C.~Bu.
\newblock On a generalization of the spectral {M}antel's theorem.
\newblock \emph{J. Comb. Optim.}, 46\penalty0 (2):\penalty0 Paper No. 14, 10,
  2023.
\newblock \doi{10.1007/s10878-023-01081-y}.

\bibitem[Nikiforov(2007)]{Nikiforov2007}
V.~Nikiforov.
\newblock Bounds on graph eigenvalues. {II}.
\newblock \emph{Linear Algebra Appl.}, 427\penalty0 (2-3):\penalty0 183--189,
  2007.
\newblock \doi{10.1016/j.laa.2007.07.010}.

\bibitem[Nikiforov(2017)]{Nikiforov}
V.~Nikiforov.
\newblock Merging the {$A$}- and {$Q$}-spectral theories.
\newblock \emph{Appl. Anal. Discrete Math.}, 11\penalty0 (1):\penalty0 81--107,
  2017.
\newblock \doi{10.2298/AADM1701081N}.

\bibitem[Nosal(1970)]{Nosal}
E.~Nosal.
\newblock Eigenvalues of graphs.
\newblock Master's thesis, University of Calgary, 1970.

\bibitem[Qi(2005)]{Qi2005}
L.~Qi.
\newblock Eigenvalues of a real supersymmetric tensor.
\newblock \emph{J. Symbolic Comput.}, 40\penalty0 (6):\penalty0 1302--1324,
  2005.
\newblock \doi{10.1016/j.jsc.2005.05.007}.

\bibitem[Qi(2013)]{Qi2013}
L.~Qi.
\newblock Symmetric nonnegative tensors and copositive tensors.
\newblock \emph{Linear Algebra Appl.}, 439\penalty0 (1):\penalty0 228--238,
  2013.
\newblock \doi{10.1016/j.laa.2013.03.015}.

\bibitem[Tur\'an(1941)]{Turan}
P.~Tur\'an.
\newblock On an extremal problem in graph theory.
\newblock \emph{Mat. Fiz. Lapok}, 48:\penalty0 436--452, 1941.

\bibitem[Wang(2020)]{Wang}
J.~Wang.
\newblock The shifting method and generalized {T}ur\'an number of matchings.
\newblock \emph{European J. Combin.}, 85:\penalty0 103057, 7, 2020.
\newblock \doi{10.1016/j.ejc.2019.103057}.

\bibitem[Yan et~al.(2026)Yan, Yang, and Peng]{YanYangPeng}
Z.~Yan, B.~Yang, and Y.~Peng.
\newblock A spectral generalized {A}lon-{F}rankl theorem.
\newblock \emph{Discrete Math.}, 349\penalty0 (2):\penalty0 Paper No. 114785,
  12, 2026.
\newblock \doi{10.1016/j.disc.2025.114785}.

\bibitem[Yu and Peng(2025)]{YuPengZykov}
L.~Yu and Y.~Peng.
\newblock A spectral version of the theorem of {Zykov} and {Erd{\H{o}}s}.
\newblock \emph{Electron. J. Combin.}, 32\penalty0 (4):\penalty0 Paper No.
  4.15, 11, 2025.
\newblock \doi{10.37236/14119}.

\bibitem[Yu and Peng(2026)]{YuPeng}
L.~Yu and Y.~Peng.
\newblock A spectral generalized {Erd{\H{o}}s--Gallai} theorem.
\newblock \emph{Discrete Math.}, 349\penalty0 (1):\penalty0 Paper No. 114672,
  7, 2026.
\newblock \doi{10.1016/j.disc.2025.114672}.

\bibitem[Yuan and Shao(2023)]{YuanShao}
X.~Yuan and Z.~Shao.
\newblock On the maximal {$\alpha $}-spectral radius of graphs with given
  matching number.
\newblock \emph{Linear Multilinear Algebra}, 71\penalty0 (10):\penalty0
  1681--1690, 2023.
\newblock \doi{10.1080/03081087.2022.2071412}.

\bibitem[Zykov(1949)]{Zykov}
A.~A. Zykov.
\newblock On some properties of linear complexes.
\newblock \emph{Mat. Sbornik N.S.}, 24/66:\penalty0 163--188, 1949.

\end{thebibliography}
\endgroup

\end{document}